\documentclass[11pt]{amsart}
\usepackage{graphicx}
\usepackage{amssymb}
\usepackage{epstopdf}

\usepackage{amsmath,amsfonts,amsthm,amsopn,cite,mathrsfs}

 \theoremstyle{theorem}
\newtheorem{lemma}{Lemma}
\newtheorem{theorem}{Theorem}

\newtheorem*{lfact}{Lemma A}

\theoremstyle{remark}

\newcommand{\DD}{\mathbb D}
\newcommand{\TT}{\mathbb T}

\newcommand{\vf}{\varphi}

\newcommand{\vt}{\vartheta}

\newcommand{\eps}{\epsilon}

\newcommand{\beq}{\begin{equation}}
 \newcommand{\eeq}{\end{equation}}

\begin{document}

\dedicatory{Dedicated to Nikolai K. Nikolski  on the occasion of his 70th birthday}

\title{Composition operators on model spaces}
 
\author{Yurii I. Lyubarskii}
\address{Department of Mathematical Sciences, Norwegian University of Science and Technology, NO-7491, Trondheim, Norway}
\email{yura@math.ntnu.no} 
\author{ Eugenia Malinnikova}
\email{eugenia@math.ntnu.no}

\subjclass[2010]{ Primary 47B33; Secondary 30H10, 30J05, 47A45.}
\keywords{Composition operator, model space, Nevanlinna counting function, Aleksandrov-Clark measure, one-component inner function}
\begin{abstract}
Let $\vf:\DD\rightarrow\DD$ be a holomorphic function, $\vt:\DD\rightarrow\DD$ be an inner function and $K_\vt(\DD)=H^2(\DD)\ominus\vt H^2(\DD)$ be the corresponding model space. We study the composition operator $C_\vf$ on $K_\vt$ and give a necessary  and sufficient  condition for $C_\vf:K_\vt\rightarrow H^2$ to be compact. The condition involves an interplay between $\vt$ and the Nevanlinna counting function of $\vf$. For a one-component $\vt$ a characterization of  compact composition operators $C_\vf$  in terms of the Aleksandrov-Clark measures of $\vf$ and the spectrum of $\vt$ is also given.  
\end{abstract}
\thanks{The authors were partly supported   by the Research
Council of Norway grants 213638/F20 and 213440/BG.}
\maketitle

\section{Introduction}
 
Let $\DD$ be the unit disk in the complex plane.  Given a holomorphic function $\vf: \DD \to \DD$, denote by
 \[
 C_\vf: f \mapsto f\circ \vf
 \] 
 the  composition operator defined on holomorphic functions  in $\DD$.  This operator is bounded on the Hardy space 
 $H^2(\DD)$ (see e.g. \cite{S}).  One of the intensively studied  questions is when $C_\vf$ is a compact operator on various spaces of analytic functions. 
 We refer the reader to the monographs \cite{Shbook, CMc} for the history and basic results on composition operators.
 
 Loosely speaking $C_\vf$ is compact on $H^2(\DD)$ if $\vf(z)$ does not approach the unit circle $\TT$ too fast as $z\to \TT$.
 J.~Shapiro \cite{S}  quantified this idea by using the Nevanlinna counting function 
 \[
 N_\vf(w)= \sum_{\vf(z)=w} -\log|z|.
 \] 
 He proved in particular that $C_\vf$ is compact on $H^2(\DD)$ if and only if
 \beq
 \label{compactness}
 \lim_{|w|\rightarrow 1-}N_\vf(w)/(-\log|w|)=0.
 \eeq
 The basic tool in his argument  is the Stanton formula
 \beq
 \label{normcomposition}
 \|C_\vf f\|^2= 2 \int_\DD |f'(z)|^2 N_\vf(z) dA(z) + |f(\vf(0))|^2,
 \eeq
where $A$ is the normalized area measure. It is obtained from the identity
 \beq
 \label{stanton}
  \|f\|^2=2\int_\DD |f'(z)|^2\log \frac 1{|z|^2}dA(z)+|f(0)|^2, \ f\in H^2(\DD),
 \eeq
by substituting $f\circ\vf$ in place of $f$. 

 Another way to describe the compactness property of $C_\vf$ is related to the Aleksandrov-Clark measures of $\vf$. These are the positive measures 
 $\mu_\alpha$ on $\TT$ defined by the relation 
 \[ 
 \Re \frac {\alpha+ \vf(z)}{\alpha-\vf(z)} 
 = \int_\TT P_z d\mu_\alpha,
 \]
where $P_z$ is the Poisson kernel, $\alpha\in\TT$. We refer the reader to the surveys \cite{PS, Saks} for more details.  In \cite{Sar} D. Sarason showed how $C_\vf$ can be treated as an integral operator on the spaces $L^1(\TT)$ and $M(\TT)$ and proved that $C_\vf$ is compact on these spaces if and only if each $\mu_\alpha$ is absolutely continuous. Due to  \cite{SSa},  it is further equivalent to $C_\vf$ being compact on $H^2(\DD)$  as well as on other Hardy spaces $H^p(\DD)$, see also \cite{CM}.

In this article  we study the  compactness of the operator    $C_\vf : K_\vt \to H^2(\DD)$, where $\vt$ is an inner function in $\DD$ and $K_\vt= H^2(\DD)\ominus \vt H^2(\DD)$ is the corresponding model space.  Consider the canonical factorization of $\vt$     
 \[
\vt(z)=B_\Lambda(z)  \exp \left ( \int_\TT \frac {\xi+z}{\xi-z} d\omega (\xi)  \right ),
\]
where $\Lambda$ is the zero set of $\vt$, $B_\Lambda$ is the corresponding Blachke product, and $\omega$
is a singular measure on $\TT$.  Functions in $K_\vt$ admit analytic continuation through $\TT \setminus  \Sigma(\vt)$, 
where 
\[
\Sigma(\vt) =  \left (\TT\cap {\rm Clos}(\Lambda) \right ) \cup {\rm supp}(\omega)
\]
is the {\em spectrum}  of $\vt$ (see \cite{Nik}, Lecture 3). 
Therefore the compactness property of $C_\vf$ does not suffer as the values of $\vf$ approach  points in $\TT\setminus \Sigma(\vt)$.  We quantify this idea below and give a condition that is necessary and sufficient  for the compactness of  $C_\vf:K_\vt\to H^2(\DD)$. 


\subsection*{ Acknowledgments} This work was started when the authors visited the Mathematics Department of the University of California, Berkley. It is our pleasure to thank the Department for the hospitality 
and Donald Sarason for useful discussions.

The authors will also thank Anton Baranov for his comments on the preliminary version of this note and for showing us the inequality in \cite{C2} that is crucial for the proof of Theorem 1.


\section{Nevanlinna counting function}
In this section we give a counterpart of the condition (\ref{compactness}) for the operator $C_\vf:K_\vt\rightarrow H^2(\DD)$. The proof follows the ideas of \cite{S}. 
The main new ingredients are estimates for the reproducing kernels and its derivatives given in Lemma \ref{l1}  below.

Let $\kappa_w$ be  the reproducing kernel for $K_\vt$,
$$
\kappa_w(\zeta)= \frac{1-{\overline{ \vt(w)} }    \vt(\zeta)} {1-\bar w   \zeta}, \quad  
\|\kappa_w\|^2 = \frac{1-|\vt(w)|^2}{1-|w|^2},
$$
and let $\tilde{\kappa}_w$ be its normalized version
 \[
\tilde{\kappa}_w(\zeta)=  \left ( \frac{1-|w|^2} {1-|\vt(w)|^2} \right ) ^{1/2}   
  \frac{1-{\overline{ \vt(w)}} \vt(\zeta)}{1-\bar w \zeta}.
\]
By
 \beq
 \label{kernel}
 k_w(\zeta)= \frac 1{1-\bar{w}\zeta}, \quad  \tilde{k}_w(\zeta)= \frac{ (1-|w|^2)^{1/2}}{1-\bar{w}\zeta} 
 \eeq
 we denote the reproducing  kernel for $H^2$ and its normalized version.

\begin{lemma} \label{l1}
Let $\{w_n\}\subset \DD$, $|w_n|\to 1$ be such that
\beq
\label{thetaestimate}
|\vt(w_n)|<a,
\eeq
 for some $a\in (0,1)$. Then\\
 (i) $\tilde{\kappa}_{w_n}\xrightarrow{w*} 0$    as $n\to \infty$; \\
(ii) there  exist  $\eps >0$ ,  $c>0$ and $n_0$ such that  
 \beq
 \label{estimatefrombelow}
| \kappa'_{w_n}(\zeta)|>\frac c {(1-|w_n|^2)^2}, \quad  \zeta \in D_{\eps}(w_n)
\eeq
holds for any $n>n_0$,  where $D_\eps(w)= \{\zeta; |\zeta-w|<\eps |1-\bar{\zeta}w|\}$ is a hyperbolic disk with center at $w$.
\end{lemma} 
 
 \begin{proof}
 (i) It suffices to show that 
\[
  \frac{(1-|w_n|^2)^{1/2} }{ 1-\bar {w}_n \zeta}      (1-\overline{ \vt(w_n)} \vt(\zeta)) \xrightarrow{w*} 0 \ {\mbox{ in} }\  L^2(\TT)
  \ {\mbox{as}}\  \ |w_n|\to 1 .   
\]  
 This in turn follows from the known fact that the normalized reproducing kernels $\tilde{k}_{w_n}$ for the Hardy space $H^2(\DD)$ tend weakly to 0 as $|w_n|\rightarrow 1$, see e.g. \cite{S}.

 (ii)  We start with  the following well-known estimate
\beq
\label{wellknown} 
|\vt'(\zeta)| \leq \frac {1-|\vt(\zeta)|^2}{1-|\zeta|^2}, \quad \zeta \in \DD.
\eeq
Together with (\ref{thetaestimate}) it readily yields 
 \beq
\label{nice}
|\vt(\zeta)|< b,\quad \zeta \in\cup_n D_\eps(w_n),
\eeq
for some $b<1$ and $\eps >0$.

We claim now that  for sufficiently large $n_0$
 \beq
 \label{central}
 |\kappa'_{\zeta}(\zeta)|> \frac {{\rm const}} {(1-|\zeta|^2)^2}, \quad \zeta\in\cup_{n>n_0} D_\eps(w_n).
 \eeq
 Indeed,  
\[
\kappa_\zeta'(\zeta)= -\frac{\vt'(\zeta)\overline{\vt(\zeta)}}{1-|\zeta|^2} + \bar{\zeta} \frac {1-|\vt(\zeta)|^2}{(1-|\zeta|^2)^2}= A_1+A_2.
\]
It follows from \eqref{nice}  that $|A_2|> c (1-|\zeta|^2)^{-2}$ for some $c>0$, and in order to prove \eqref{central}
it suffices to show that 
\[  
|A_1|<q|A_2| \quad   \mbox{for some} \ q\in (0,1),
\]
when  $\zeta\in\cup_{n>n_0}D_\eps(w_n)$. The relation \eqref{wellknown}  yields
 \[
|A_1|\le |\vt(\zeta)|  \frac {1-|\vt(\zeta)|^2}{(1-|\zeta|^2)^2}<\frac{b}{|\zeta|}|A_2|,\quad  \zeta\in\cup_n D_\eps(w_n).
\]
Since $b<1$ and  $\inf\{|\zeta|: \zeta\in\cup_{n>m} D_{\eps(w_n)}\}\to 1$ as $m\to \infty$, the required estimate follows.

The inequality \eqref{central} proves (\ref{estimatefrombelow}) for the special case $\zeta=w_n$. In order to complete the proof consider the function
\[
g(w,\zeta) = \overline{\kappa'_w(\zeta)}  = - \frac{\overline{\vt'(\zeta)}\vt(w)}{1-\bar{\zeta}w} + w \frac{1-\overline{\vt(\zeta)}\vt(w)}{(1-\bar{\zeta} w)^2}.
\]
  We have 
$|g(w_n,\zeta)|= |\kappa'_{w_n}(\zeta)|.$
On the other hand 
\beq
\label{lagrange}
|g(w_n,\zeta)-g(\zeta,\zeta)|<|g'(\tilde{w}, \zeta)||\zeta-w_n|,
\eeq
for some point $\tilde{w}\in [\zeta,w_n]$, where the derivative is taken with respect to the first variable. A straightforward estimate shows
\[
|g'(\tilde{w}, \zeta)| < \frac {\mbox{const}}{(1-|w_n|)^3}, \ \tilde{w}, \zeta \in D_\eps(w_n),
\]
the constant being independent of $n$. Now, given any $\eta >0$ we can choose $\eps'<\eps$ such that 
  the right-hand side  in \eqref{lagrange} does not exceed  $\eta (1-|\zeta|^2)^{-2}$ when $\zeta\in D_{\eps'}(w_n)$.
Taking $\eta$ sufficiently small we obtain (\ref{estimatefrombelow}).
\end{proof}

In what follows we  assume for simplicity that $\vf(0)=0$.
\begin{theorem}\label{th:1} The following statements are equivalent\\
(C) $ C_\vf:K_\vt \to H^2$ is a compact operator.\\
(N) The Nevanlinna counting function of $\vf$ satisfies 
\beq
 \label{basic}
   N_\vf(w) \frac {1-|\vt(w)|^2}{1- |w|^2}  \to 0 \ {\rm as}  \  |w|\nearrow 1.     
\eeq
 \end{theorem} 

\begin{proof}[Proof (N)  $\Rightarrow$   (C)] 
Since $N_\vf(w)(1-|(w)|^2)^{-1}$ and $1-|\vt(w)|^2$ are bounded, the condition (N) means that for any $a<1$ 
\[
\lim_{|\vt(w)|<a, |w|\rightarrow 1-} N_\vf(w)(1-|w|^2)^{-1}=0.\]
In particular, for any $p>0$
\[
   N_\vf(w) \frac {(1-|\vt(w)|)^p}{1- |w|}  \to 0 \ {\rm as}  \  |w|\nearrow 1.     
\]

 We use the following  inequality, see \cite[page 187]{C2} and \cite{ACS},
 \begin{equation}
\label{cohn}
 \|f\|^2  \ge C_p  \int_\DD |f'(z)|^2   \frac{1-|z|}{(1-|\vt(z)|)^p}dA(z)+ |f(0)|^2 ,\quad  f \in K_\vt,
 \end{equation}
which is valid for some $p \in (0,1)$.

In our setting this formula replaces \eqref{stanton}. We follow the argument of J.~Shapiro; 
a similar argument for compactness of the composition operator in some weighted spaces of analytic functions can be found in \cite[Ch. 3.2]{CMc}.

Let  $K^{(n)}_\vt=\{ f\in K_\vt; f  \ \mbox{has  zero of order $n$ at the origin} \}$,  and let  
$\Pi^{(n)}:K_\vt \to K^{(n)}_\vt$  be the corresponding orthogonal projection.
We will prove that 
\[
\|C_\vf \Pi^{(n)}\|_{K_\vt \to H^2} \to 0, \quad n\to \infty.
\]
Thus $C_\vf$ is compact as it can be approximated by the finite-rank operators $C_\vf (I-\Pi^{(n)}) $.

Indeed, given $f\in K_\vt$, $\|f\|=1$, denote $g_n= \Pi^{(n)}f$. We have $\|g_n\|\leq 1$  and, for each $R<1$, $\eps >0$ we can choose 
$ n(\eps,R)$  independent of $f$ and such that  
\[
|g_n(w)|< \eps,  \  |g'_n(w)| < \eps,  \ \mbox{for all} \ n>n(\eps, R), \ \mbox{and} \ |w|<R.   
\]
It follows from \eqref{cohn} that 
\[
\int_\DD |g_n'(z)|^2   \frac{1-|z|}{(1-|\vt(z)|)^p}dA(z) < C,
\]
with $C$ independent  of $f$, $n$. Next, by (\ref{normcomposition}) we have
\begin{multline*}
\|C_\vf  \Pi^{(n)}f\|^2  =\int_\DD |g_n'(z)|^2   N_\vf(z) dA(z) \leq \int_{|z|<R} + \int_{R<|z|<1} \leq \\
\max_{|z|<R}\left \{ |g'_n(z)|^2 \right \} \int_{|z|<R}  N_\vf(z) dA(z) + \\
                      \max_{R<|z|<1} \left \{ N_\vf(z) \frac{(1-|\vt(z)|)^p}{1-|z|}  \right \} 
                                 \underbrace {   \int_{R<|z|<1}|g'_n(z)|^2 \frac{1-|z|}{(1-|\vt(z)|)^p} dA(z) }_
                                           {\mbox {\tiny{this is less than }} C}.
 \end{multline*}

Choosing first $R$ such that the second summand is small, and then $n$ large enough  to provide smallness of the 
first summand we can make the whole expression arbitrary small for all $f\in K_\vt$, $\|f\|=1$.  
\end{proof}

\begin{proof}[ Proof (C)  $\Rightarrow$  (N)] Assume that $C_\vf$ is compact but (\ref{basic}) does not hold. Then there exists a sequence $\{w_n\} \subset \DD$, $|w_n|\to 1$,  satisfying 
 \beq
\label{not}
N_\vf(w_n) \frac{1-|\vt(w_n)|^2}{1-|w_n|^2}>c>0.
\eeq
 By  the Littlewood subordination principle, which implies that $N_\vf(w)\leq \log\frac{1}{|w|}$, there exists $a<1$ such that (\ref{thetaestimate}) holds. Applying Lemma \ref{l1} (i) and the compactness of $C_\vf$, we get 
 $\|C_\vf\tilde{\kappa}_{w_n}\|^2 \to 0 \  \mbox{as}  \ n\to \infty.$
On the other hand, (\ref{thetaestimate}), Lemma \ref{l1} (ii) and  the subharmonicity inequality for $N_\vf$  (see \cite{S}) imply  
\begin{multline}
\label{finalestimate}
\|C_\vf\tilde{\kappa}_{w_n}\|^2 \ge \int_\DD |\tilde{\kappa}_{w_n}'(\zeta)|^2 N_\vf(\zeta) dA(\zeta)\ge\\
c_1\int_{\DD}|\kappa_{w_n}'(\zeta)|^2(1-|w_n|^2) N_\vf(\zeta)dA(\zeta)\ge
               \frac {c_2} {(1-|w_n|^2)^3}
                                       \int_{D_\eps(w_n)}  N_\vf(\zeta) dA(\zeta) \ge\\ \frac{c_\eps N_\vf(w_n)}{1-|w_n|^2}.
 \end{multline}                      
We combine the last estimate with (\ref{not}) to get a contradiction.
 \end{proof}


\section{Aleksandrov-Clark measures}

For $\alpha \in \TT$ let as before $\mu_\alpha$ be the Aleksandrov-Clark measure of $\vf$ corresponding to $\alpha$  and let
\[
d\mu_\alpha= h_\alpha dm + d\sigma_\alpha
\]
be its decomposition into absolutely continuous and singular parts, where $m$ is the normalized Lebesgue measure on $\TT$.
Then
\[
h_\alpha(\zeta)=\frac{1-|\vf(\zeta)|^2}{|\alpha-\vf(\zeta)|^2}
\]
for almost every $\zeta$ on $\TT$. As above, we assume for simplicity that $\vf(0)=0$, then $\|\mu_a\|=1$.

We give a condition in terms of the Aleksandrov-Clark measures that is sufficient for the compactness, it is also necessary if
$\vt$ is a one-component inner function, i.e. the set  $\{z\in \DD: |\vt(z)|<r\}$ is connected for some $r\in (0,1)$.The  one-component inner functions were introduced by W.~S.~ Cohn  in \cite{C}, 
see also \cite{A} for a number of equivalent characterizations of one-component inner functions.

\begin{theorem} Let  $\vt$ be a  one-component inner function. The following statements are equivalent\\
(C) $ C_\vf:K_\vt \to H^2$ is a compact operator.\\
 (S)  $\sigma_\alpha=0$ for all $\alpha \in \Sigma(\vt)$.\\
Moreover, the implication (S) $\Rightarrow$  (C) holds for any inner function $\vt$.  
\end{theorem} 

The proof mainly follows the  pattern as described in \cite{Saks}  section 7, see \cite{Sar} for the original approach and also\cite{CM}. 
We need the following description of the spectrum of a one-component inner function.

 \begin{lfact} 
 Let $\vt$ be a one-component inner function  and $\alpha \in \TT$. The following statements are equivalent \\ 
(a)  $\alpha \in \Sigma(\vt)$; (b) $\liminf_{w\rightarrow\alpha} |\vt(w)| <1$;  (c)  $\liminf_{r\rightarrow 1-} |\vt(r\alpha)| <1$.
\end{lfact}
 
The implications $(c)\Rightarrow (a)\Rightarrow (b)\Rightarrow (a)$ are straightforward and hold for any inner function, see \cite{Nik}, Lecture 3; $(a)\Rightarrow (c)$  is true when $\vt$ is one-component, it follows from   \cite{VT}, Section 5, see also Theorem 1.11 in \cite{A}.

\begin{proof}[Proof (C) $\Rightarrow$ (S)]  Fix $\alpha \in \Sigma(\vt)$ and chose a sequence $r_n\to 1 $ so that 
 $|\vt(\alpha r_n)|<a<1$. By Lemma \ref{l1}, we have
 \[
 \left \|  C_\vf \tilde{\kappa}_{\alpha r_n}\right \|^2 \ge
     \int_\TT \frac{1-r_n^2}{|1-\bar{\alpha} r_n \vf(\xi)|^2}
            \frac {|1-\bar{\vt}(\alpha r_n) \vt(\vf(\xi))|^2}{1- |\vt(\alpha r_n)|^2} |d\xi| \ \to \ 0, \ {\rm as} \ n\to \infty.
 \]
Since $|\vt(\alpha r_n)|<a<1$, this  yields
  \[
 \|  C_\vf \tilde{k}_{\alpha r_n} \|^2 = 
          \int_\TT \frac {1-r_n^2}{|1-\bar{\alpha}r_n \vf(\xi)|^2}|d\xi| \to 0, \ {\rm as} \ n\to \infty,
 \]
 where $\tilde{k}_w$ is the normalized reproducing kernel for $H^2$, see (\ref{kernel}).

 The rest of the proof follows literally \cite{CM}, see also \cite{Saks}, Lemma 7.6. We give it here for
 the sake of completeness. We have
 \[
\|C_\vf\tilde{k}_{\alpha r_n} \|^2=
\int_{\TT} \frac{1-|r_n\vf(\xi)|^2}{|\alpha-r_n\vf(\xi)|^2}|d\xi| - \int_{\TT}r_n^2
          \frac{1-|\vf(\xi)|^2}{|\alpha-r_n\vf(\xi)|^2}|d\xi| =: A_n-B_n.
  \]
 Clearly,
\[A_n=\Re\left(\frac{\alpha+r_n\vf(0)}{\alpha-r_n\vf(0)}\right)=1.\] 
Further, by the monotone convergence theorem  
  \[
  \lim_{n\rightarrow \infty}B_n =\int_\TT  \frac{1-|\vf|^2}{|\alpha-\vf|^2}=\|h_\alpha\|_1, \  {\rm as} \ n\to \infty.
  \]
Then $\|\sigma_\alpha\|=1-\|h_\alpha\|_1=0$. This completes the proof (C) $\Rightarrow$ (S).  
\end{proof}
  
We remark that  the one-component condition was employed only in the description of the spectrum, so the following statement holds for any inner function:
{\em If $C_\vf:K_\vt\rightarrow H^2$ is a compact operator then $\sigma_\alpha=0$ for all $\alpha\in\TT$ such that $\liminf_{r\rightarrow 1-}|\vt(r\alpha)|<1$.}  

\begin{proof}[ Proof (S) $\Rightarrow$ (N)] We will prove this implication and refer to Theorem \ref{th:1}.  Suppose that (N) is false, then 
\[
N_\vf(w_n) \frac{1-|\vt(w_n)|^2}{1-|w_n|^2} > c>0, \ {\rm for \ some} \  \{w_n\}\subset \DD, w_n\rightarrow\alpha\in\TT.
\]  
Clearly $\alpha\in\Sigma(\vt)$ and  \eqref{thetaestimate}  holds for some $a<1$. Further, \[(1-|w_n|)^{-1}N_\vf(w_n)>c_1>0.\] 
We  obtain a contradiction in the same way as in \cite{CM} see also \cite{Saks}, Theorem 7.5.   We have, by a simple version of (\ref{finalestimate})
  \[
\|C_\vf\tilde{k}_{w_n}\|^2\ge C\frac{N_\vf(w_n)}{1-|w_n|^2}>c_2>0.\]
On the other hand by the Fatou lemma, 
\begin{multline*}
\limsup_{n\rightarrow\infty}\|C_\vf\tilde{k}_{w_n}\|^2=1 - \liminf_{n\rightarrow\infty}\int_\TT|w_n|^2 \frac{1-|\vf(\xi)|^2}{|1-\bar{w}_n\vf(\xi)|^2} |d\xi|\le\\
1 -\int_\TT\frac{1-|\vf(\xi)|^2}{|\alpha-\vf(\xi)|^2}|d\xi|=\|\sigma_\alpha\|=0,
 \end{multline*}
 which leads to a contradiction. 
\end{proof}


\section{Examples and concluding remarks}

 {\em Inner functions with one point spectra.} Consider the Paley-Wiener space $K_{\vt_1}$ generated by 
 \[
 \vt_1(z)=e^{\frac{z+1}{z-1}},
 \]
this space can be obtained from the classical Paley-Wiener space of entire functions by the substitution 
$\zeta \mapsto \frac{z-i}{z+i}$. Then $\vt_1$ is a one-component inner function and $\Sigma(\vt_1)=\{1\}$.  
Theorem \ref{th:1} and explicit calculation show that $C_\vf:K_{\vt_1}\to H^2$ is a compact operator if and only if
\beq
\label{PW}
\frac{N_\vf(w)}{\max  \left ((1-|w|^2),|1-w|^2 \right )} \ \to \ 0, \ \ {\rm as}  \ \ w\to 1.
\eeq

Consider  now $D=\{w\in \DD; |w-1/4|<3/4\}$  and let $\vf$  be  a conformal mapping $\vf  : \DD  \to D $, $\phi(0)=0, \vf(1)=1$. 
Clearly  (\ref{PW}) does not hold and the operator   $C_\vf:K_{\vt_1}\to H^2$ is not compact. Evidently,
the Aleksandrov-Clark measure $\mu_1$ of $\vf$ is not absolutely continuous.   Below we give an example of  a (multi-component) inner function 
$\vt$  with $\Sigma(\vt)=\{1\}$ and such that $C_\vf :  K_\vt \to H^2$ is a compact operator.  Thus (C) does not imply (S) for general $\vt$. 

Take a sequence $t_m\searrow  0$ such that   $\{\zeta_m\}=\{(1-t_m^3) ^{1/2}e^{it_m}\}$ is 
  an interpolating sequence in $\DD$. Given a sequence $\{\alpha_m\}\in l^1,\ \alpha_m\in(0,1),$ denote  
  $\Lambda=\{\lambda_m\}=\{(1-\alpha_mt_m^3)^{1/2}e^{it_m}\}$, this is also an interpolating sequence. Let now 
  $\vt=B_\Lambda$ be the Blaschke product corresponding to the sequence $\Lambda$. We claim that 
  $C_\vf:K_\vt \to H^2$  is a compact operator.
 
  Indeed, $\|C_\vf \tilde{k}_\zeta\|\leq 1$, $\zeta \in \DD$, here $\tilde{k}_\zeta$ is the normalized reproducing kernel for $H^2$, 
  this follows   just from the fact that  $C_\vf$ is contractive.  In particular
  \[
   t_m^3 \int_\TT \frac {|d\xi|}{|1-\bar{\zeta}_m \vf(\xi)|^2}=\|C_\vf \tilde{k}_{\zeta_m}\|^2 \le 1.
  \]
  Since $|1-\bar{\zeta}_m \vf(\xi)|^2\le c |1-\bar{\lambda}_m \vf(\xi)|^2$, $\xi \in \TT$, we have,
  \[
  \|C_\vf \tilde{k}_{\lambda_m}\|^2 \asymp \alpha_m t_m^3 \int_\TT \frac {|d\xi|}{|1-\bar{\lambda}_m \vf(\xi)|^2} 
  \le C\alpha_m.
 \]
 On the other hand the system $\{\tilde{k}_{\lambda_m}\}$ forms a Riesz basis in $K_\vt$ (see e.g. \cite{Nik}, Lecture VII).
 Compactness of $C_\vf:K_\vt \to H^2$ is now straightforward, alternatively it could be deduced from Theorem 1.

\medskip

{\em Concluding remarks.} 
In the classical case of $H^2(\DD)$  the essential norm of the composition operator was obtained by J.~Shapiro
\[
\|C_\vf\|_e^2=\limsup_{|w|\rightarrow 1-}\frac{N_\vf(w)}{-\log|w|}.
\]
For a given one-component $\vt$ the equivalence of the norms proved in \cite{C2} and similar arguments give 
\[
\|C_\vf : K_\vt\rightarrow H^2\|_e^2\asymp \limsup_{|w|\rightarrow 1-}N_\vf(w)\frac{1-|\vt(w)|^2}{1-|w|^2}.
\]

Let $\vf:\DD\rightarrow\DD$ be a holomorphic function and 
$\vf^*$ be its radial boundary values . Define a measure $\nu_\vf$ on $\bar{\DD}$ 
by $\nu_\vf(E)=m((\vf^*)^{-1}(E))$ for any $E\subset \bar{\DD}$, where $m$ is 
the Lebesgue measure on $\TT$. The composition operator $C_\vf$ on $H^2(\DD)$  is isometrically equivalent to the embedding of $H^2$ into 
$L^2(\bar{\DD},\nu_\vf)$, see \cite{Mc, CMc} for details. The connecting between the Nevanlinna counting function and the measure $\nu_\vf$ was recently 
studied in details in \cite{LQLR}.

Respectively, the compactness of the composition operator on $K_\vt$ can be reduced 
to the question of the compactness of the embedding $K_\vt\hookrightarrow L^2(\bar{\DD},\nu_\vf)$. 
It is well-known that the embeddings are easier to study for one-component inner functions $\vt$, see
 \cite{C,C2}, and  subsequent works \cite{ VT} and \cite{A,A1}. 
 The  compactness of the embedding $K_\vt\hookrightarrow L^2(\bar{\DD},\mu)$  was studied by J.~A.~Cima and A.~L.~Matheson \cite{CM1} and by A.~D.~Baranov  \cite{B}. The latter article contains in particular 
 necessary and sufficient conditions for the compactness of the embedding 
 for the case of one-component inner function.  This approach also shows that for one-component 
   $\vt$ the compactness of the composition operator $C_\vf:K_\vt^ p\rightarrow H^p$ 
   does not depend on $p\in(1,\infty)$.

\end{document}